\documentclass[letterpaper, 10 pt, conference]{ieeeconf} 
\IEEEoverridecommandlockouts                            
\usepackage{amsmath,amssymb,amsfonts,mathtools, amsthm}
\usepackage{graphicx}
\usepackage{cite}  
\usepackage{subfig}
\usepackage{subcaption}
\usepackage{dsfont}
\usepackage{comment}
\usepackage{xcolor} 
\usepackage{balance}

\newtheorem{assumption}{Assumption}
\newtheorem{remark}{Remark}
\newtheorem{definition}{Definition}
\newtheorem{proposition}{Proposition}

\newtheorem{lemma}{Lemma}
\newtheorem{theorem}{Theorem}
\usepackage{algorithm}
\usepackage{algpseudocode}

\allowdisplaybreaks

\DeclareMathOperator*{\minimize}{min}
\DeclareMathOperator{\st}{s.t.}
\DeclareMathOperator{\col}{col}

\title{\LARGE \bf
 Adversarially and Distributionally Robust Virtual Energy Storage Systems via the Scenario Approach}

\author{
Georgios Pantazis,
Nicola Mignoni,
Raffaele Carli, 
Mariagrazia Dotoli, 
Sergio Grammatico
\thanks{The work of N. Mignoni, R. Carli, and M. Dotoli was supported by the NRRP - Mission 4 Component 2 Investment 1.3 - Call for tender No. 341 of March 15, 2022 of MUR - project ``NEST (Network 4 Energy Sustainable Transition)'' (project no. PE00000021) and NRRP - Mission 4 Component 2 Investment 1.1  - Call for ``Projects of significant national interest - Progetti di rilevante interesse nazionale (PRIN)'' - Decree no. 104 of February 2, 2022 of MUR (project: CORRECT, code: 202248PWRY).}
\thanks{G. Pantazis is with the Dynamics and Control group of Eindhoven University of Technology, The Netherlands (e-mail: {\tt\small{g.pantazis33@tue.nl}}). N. Mignoni, R. Carli, and M. Dotoli are with the Department of Electrical and Information Engineering of the Polytechnic of Bari, Italy (e-mail: {\tt\small\{nicola.mignoni, raffaele.carli, mariagrazia.dotoli\}@poliba.it}). S. Grammatico is with the Delft Center for Systems and Control of the Technical University of Delft, Delft, the Netherlands (e-mail: {\tt\small S.Grammatico@tudelft.nl}).}
}

\begin{document}
\maketitle
\thispagestyle{empty}
\pagestyle{empty}

\begin{abstract}
We study virtual energy storage services based on the aggregation of EV batteries in parking lots under time-varying, uncertain EV departures and state-of-charge limits. We propose a convex data-driven scheduling framework in which a parking lot manager provides storage services to a prosumer community while interacting with a retailer. The framework yields finite-sample, distribution-free guarantees on constraint violations and allows the parking lot manager to explicitly tune the trade-off between economic performance and operational safety. To enhance reliability under imperfect data, we extend the formulation to adversarial perturbations of the training samples and Wasserstein distributional shifts, obtaining robustness certificates against both corrupted data and out-of-distribution uncertainty. Numerical studies confirm the predicted profit-risk trade-off and show consistency between the theoretical certificates and the observed violation levels.
\end{abstract}

\section{Introduction}
Repurposing electric vehicle (EV) charging facilities in parking lots as \emph{virtual energy storage systems} (VESS) allows the distribution system operator (DSO) to leverage EV aggregation for the provision of services to prosumers, which can enhance stability and reduce costs of purchasing physical storage facilities from the prosumers' side \cite{mignoni2023noncooperative, Zanvettor_Fochesato_2024}. At the same time, recent advances in optimization under uncertainty can assist parking lot managers with making better-informed decisions.  Though early studies have shown that fleets of EVs can emulate dispatchable resources when state-of-charge dynamics and availability are respected,  most rely on deterministic or parametric uncertainty models that do not take into account day-to-day variability in arrivals, departures, and user preferences \cite{Sevdari2022,Zanvettor2022}. This motivates the recent shift towards data-driven, distribution-free formulations with explicit finite-sample guarantees for out-of-sample feasibility and performance \cite{data-driven-uc,robust-ev-scheduling}. 

Scenario optimization allows solving semi-infinite robust optimization problems by approximating them with a tractable program built from sampled scenarios \cite{scenario-tutorial}, \cite{CalafioreCampi2005}, \cite{CampiGaratti2008exact}. Standard results provide \emph{a priori} bounds on violation probabilities in terms of decision dimension and sample size, while more recent \emph{a posteriori} bounds leverage the number of support constraints to tighten the provided guarantees \cite{wait-and-judge,posteriori-support}.  Providing guarantees for EV charging scheduling in parking lots, leveraging the scenario approach can be found in
\cite{Paccagnan2019, Pantazis2020_cdc, Pantazis2022, Fele2019, pantazis2024apriori} where uncertain operating constraints or costs are considered. However, such solutions do not consider EV parking lots interconnected with retailers and prosumer communities and mainly focus on EV charging scheduling.  In papers that focus more on the EV parking lot management for community services, models for parking lot arbitrage and flexibility scheduling have often imposed fixed capacity envelopes or known departure processes \cite{Zanvettor2022,lot-flex-uncertain-departures}. Other methods may risk infeasibility or exhibit excessive conservatism. Only a few works model parking lots as virtual energy storage services. Specifically, \cite{Chandra2024_CoordinatedEVFleets} designs a three-stage energy management system that coordinates EV charging of fleets to maximize community benefits and operational efficiency.  The works \cite{Zheng2023} and \cite{TostadoVeliz2025} focus mainly on the market participation of the EV parking lots, e.g., the interaction between EV storage services and retailers. 

This paper studies virtual energy storage services provided through the aggregated management of EV batteries in parking lots. In this setting, three challenges naturally arise. First, the virtual storage service must be modelled at the interface of multiple actors, namely,  prosumer communities, a parking lot manager, and a retailer, while accounting for uncertain EV arrivals, departures, and user behaviour. Second, the parking lot manager must make economically meaningful decisions without sacrificing operational safety, which requires a principled and explicitly tunable trade-off between performance and constraint satisfaction. Third, the data used to characterize EV behaviour may be noisy, adversarially perturbed, or statistically shifted with respect to future operating conditions, so guarantees based solely on standard stochastic methods may not be reliable after deployment. To address these challenges we develop a methodology with the scenario approach as its backbone \cite{Campi_2020_scenario_relaxation}, \cite{JMLR_Campi}, \cite{Campi2025_DRO}. Specifically, our contributions with respect to the related literature are as follows:
\begin{enumerate}
    \item We formulate a convex model for virtual energy storage provision in which a Parking Lot Manager (PLM) aggregates parked EV batteries into a time-varying storage resource, commits storage services to a prosumer community, and interacts with a retailer through energy transactions. The resulting formulation captures uncertain EV departure losses and time-varying state-of-charge caps in a unified convex optimization framework.
    
    \item We derive a data-driven methodology based on scenario optimization that provides distribution-free out-of-sample guarantees on constraint violations. Moreover, by means of scenario relaxation \cite{Campi_2020_scenario_relaxation}, \cite{JMLR_Campi}, we endow the PLM with a tunable profit--risk mechanism with finite-sample certificates, allowing them to explicitly trade economic performance against robustness in constraint satisfaction.
    
    \item Leveraging the recent results in \cite{Campi2025_DRO}, we extend the virtual energy storage service model to account for adversarial perturbations of the training samples and to Wasserstein distributional ambiguity, thereby obtaining finite-sample robustness certificates against both corrupted data and out-of-distribution uncertainty. This extension allows the same virtual storage formulation to remain meaningful even when the uncertainty distribution at deployment differs from the one represented by the training data.
\end{enumerate}

\textit{Notation}: Sets $\mathbb{R}$ and $\mathbb{R}_{\geq 0}$ denote the real and non-negative real numbers. The $\col$ operator induces the vertical stack of its arguments. The \textit{positive part} operator is defined as $[\cdot]_+ := \max\{0, \cdot\}$, while the \textit{negative part} operator is defined as $[\cdot]_- := -\min\{0,\cdot\}$. Let $\mathcal{K}=\{1,\dots,K\}$ be a discrete time window of $K\in\mathbb{N}$ equally spaced time steps. $\mathfrak{N}(0,1)$ denotes a Gaussian distribution with mean 0 and standard deviation 1. Given a set $\mathcal{A} \subseteq \mathbb{R}^n$, $\mathds{1}_{\mathcal{A}} : \mathbb{R}^n \to \{0,1\}$ is the indicator function associated with $\mathcal{A}$, so that $\mathds{1}_{\mathcal{A}}(x) = 1$ if $x \in \mathcal{A}$ and $0$ otherwise.

\section{Problem Formulation}
\begin{figure}[t]
  \centering
  \includegraphics[width=\linewidth]{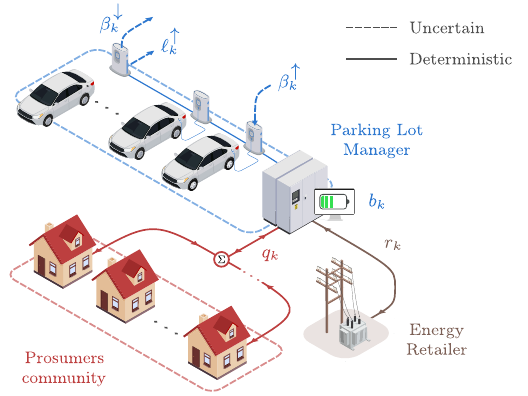} 
  \caption{ \small The parking lot manager (PLM) leverages the available storage of the parked EVs, as agreed with the EV users, to provide virtual energy storage services to a community of prosumers. Furthermore, the PLM is allowed to trade energy with retailers.   }
  \label{fig:plm-model}
\end{figure}

An overview of the considered energy community setup is illustrated in Fig. \ref{fig:plm-model}.
The virtual energy system satisfies the following constraints:
\begin{equation}
\label{eq:ess}
\mathcal{X}_k(b_0, \ell_k, \beta_k) \! = \! \left\{ \! \begin{bmatrix} r_k \\ b_k \end{bmatrix} \! \! \in \! \mathbb{R}^2 \!: \! \begin{matrix} b_k \! = \! b_{k-1}+q_k+r_k \! - \! \ell_k \\ b_k \! \in \! [0,\beta_k], \ |r_k| \! \leq \! r_{\max}
\end{matrix} \! \right\}
\end{equation}

Here, $q_k\in\mathbb{R}$ denotes the energy exchange as requested by the prosumer, which has to be met by the PLM. In particular, $q_k > 0$ ($q_k < 0$) denotes energy that the prosumers injected into (withdraw from) the virtual storage. Term $\ell_k\in \mathcal{L} \subset \mathbb{R}_{\geq 0}$ denotes the state of charge
losses due to, e.g., EVs leaving the parking lot, or operational faults. Term $\beta_k \in \mathcal{B} \subset \mathbb{R}_{\geq 0}$ is the time-varying capacity: incoming EVs increase the battery capacity, while leaving EVs reduce it. Finally, $r_k\in\mathbb{R}$ is the
PLM’s decision to buy from ($r_k > 0$) or sell ($r_k < 0$) energy to the retailer.
The need for $r_k$ serves two purposes: i) it allows the PLM to compensate state-of-charge drops due to EVs leaving, and ii) it provides a means for energy arbitrage. The second aspect is not strictly related to virtual storage markets, i.e., one could set $r_k \geq 0$. Nonetheless, we allow it for the sake of generality. 

We denote the collection of the virtual state of charge over the horizon $k \in \mathcal{K}$ by $\mathbf{b}=\col(b_k)_{ k \in \mathcal{K}}$ and $\mathbf{r}=\col(r_k)_{k \in \mathcal{K}}$, respectively.
Being an economic actor, the PLM is interested in minimizing the sustained costs, i.e., maximizing its revenues. The objective $J : \mathbb{R}^K \to \mathbb{R}$ denotes such quantity, defined as 
\begin{equation}
    \label{eq:objective}
    J(\mathbf{r}) = \sum_{k\in\mathcal{K}} \!\left(\pi_k^+[r_k]_+ + \pi_k^-[r_k]_{-}\right)
\end{equation}
where $\pi_k^+,\pi_k^- \in \mathbb{R}_{\geq 0}$ are the retailer's selling and buying price, respectively. Note that the first sum term in \eqref{eq:objective} represents the costs sustained for buying energy from the retailer, while the second represents the revenues coming from arbitrage. 
\begin{assumption}
    \label{assum:energy-prices}
    For all time steps $k\in\mathcal{K}$, the buying price is larger than the selling price, i.e., $\pi_k^+\ge \pi_k^-$ \cite{Wei2015Energy}.  \hfill $\square$
\end{assumption}
The scheduling problem of the PLM then takes the form:
\begin{equation}
        \label{eq:PLM-base}
	\minimize_{\mathbf{x}} \ J(\mathbf{r}) \ \st \ \mathbf{x}_k \in \mathcal{X}_k(b_0, \ell_k, \beta_k), \ \forall k \in \mathcal{K}
\end{equation}
where $\mathbf{x}_k = [r_k, b_k]$, so that $\mathbf{x} = \col(\mathbf{x}_k)_{k \in \mathcal{K}}$. 

In practical terms, (\ref{eq:PLM-base}) is the PLM's baseline scheduling problem, i.e., it decides how much energy to trade with the retailer and how much energy to keep in the virtual battery so as to satisfy the prosumers' request without violating storage limits.

\begin{remark}
    \label{rem:convex}
    Based on Assumption \ref{assum:energy-prices}, the objective $J(\cdot)$ is convex, and linear when $\pi^+_k = \pi^-_k$, for all $k \in \mathcal{K}$. Moreover, since the constraints in $\mathcal{X}_k(b_0, \ell_k, \beta_k)$ are affine, for all $k \in \mathcal{K}$, the problem in \eqref{eq:PLM-base} is convex. 
\end{remark}
 
We consider a general setting, where the loss vector associated with EV departures $\boldsymbol{\ell} = \col(\ell_k)_{k \in \mathcal{K}}$ and the upper bound on the virtual state of charge $\boldsymbol{\beta} = \col(\beta_k)_{k \in \mathcal{K}}$ are uncertain. Since it is generally challenging to identify the underlying distribution (if it exists) of both parameters, an approach to deal with this problem is by adopting a data-driven perspective and obtaining samples from previous values of $\beta_k$, $\ell_k$ obtained either from real measurements or data produced from a synthetic model. To this end, let $\beta^{(i)}_k$ and $\ell^{(i)}_k$ correspond to the $i$-th sample at time $k$ from a collection of samples $\mathcal{N} = \{1,\dots,N\}$ and let $\tilde{\mathcal{X}}_k \subset \mathbb{R}^2$ be the SoC-relaxed counterpart of \eqref{eq:ess}, i.e., 
\begin{equation}
\tilde{\mathcal{X}}_k(b_0, \ell_k, \beta_k) \! = \! \left\{ \! \begin{bmatrix} r_k \\ b_k \end{bmatrix} \! \! \in \! \mathbb{R}^2 \!: \! \begin{matrix} b_k \! \geq \! b_{k-1}+q_k+r_k \! - \! \ell_k \\ b_k \in [0,\beta_k], \ |r_k| \! \leq \! r_{\max}
\end{matrix} \! \right\}
\end{equation}
We can then define the following scenario approximation of the original problem:
\begin{equation}
    \label{eq:PLM-N-prime}
    \begin{aligned}
    \minimize_{\mathbf{u}, \mathbf{x}} J(\mathbf{r}) \
    \st \ &\mathbf{x}_k \in \tilde{\mathcal{X}}_k\left(b_0, u_k, \beta^{(i)}_k\right), \ u_k \geq \ell^{(i)}_k, \\ 
    &  \ \forall k \in \mathcal{K}, \forall i \in \mathcal{N},
    \end{aligned}
\end{equation}
where $u_k \in \mathbb{R}_{\geq 0}$, with $\mathbf{u} = \col(u_k)_{k \in \mathcal{K}}$ is an auxiliary decision variable. Practically, (\ref{eq:PLM-N-prime}) tells the PLM to schedule retailer trades and a protective energy buffer so that the promised virtual storage service remains feasible for all observed EV departure losses and capacity limits in the training data.

Note that the data trajectories are considered as an independent and identically distributed (i.i.d.) sample vector from an unknown probability distribution $\mathbb{P}$. This is a reasonable assumption, as in practice the pattern of vehicle departures on a given day shows very little correlation with departures on the same weekday in subsequent weeks or even in the same period of the following year.
However, correlations between components of the data trajectory can be taken into account without violating the i.i.d. assumption.
Given the program above, the probability of violation of the PLM constraints is 
\begin{equation} 
    \label{prob_viol}
    \mathbb{V}(\mathbf{r}, \mathbf{b}) := \left\{ \mathbb{P}(\tilde{\boldsymbol{\ell}}, \tilde{\boldsymbol{\beta}}): \ \exists k \in \mathcal{K} \text{ s.t. } \mathbf{x}_k \notin \tilde{\mathcal{X}}_k(b_0, \tilde{\ell}_k, \tilde{\beta}_k) \  \right\}
\end{equation}
where $\tilde{\boldsymbol{\ell}} = \col(\tilde{\ell}_k)_{k \in \mathcal{K}}$ and $\tilde{\boldsymbol{\beta}} = \col(\tilde{\beta}_k)_{k \in \mathcal{K}}$, quantifying the probability that the PLM's decision $\mathbf{r},\mathbf{b}$ will violate constraints for future yet unseen data trajectories $\boldsymbol{\tilde{\ell}}$, $\tilde{\boldsymbol{\beta}}$. From the PLM's viewpoint, (\ref{prob_viol}) measures how often a schedule that looks feasible on the training data may fail in future operation when EV departure losses or capacity limits differ from the observed scenarios.

\section{Robust Virtual Energy Storage Services}
Based on Remark \ref{rem:convex}, the problem in \eqref{eq:PLM-N-prime} is convex. Assuming feasibility, convexity alone does not ensure the uniqueness of the minimizer. To this end, we impose the following standing assumption.

\begin{assumption}[Uniqueness] \label{assum:uniqueness}
For any number of samples $N \in \mathbb{N}$ and for every sample $\boldsymbol{\ell}^{(i)} = \col(\ell^{(i)}_k)_{k \in \mathcal{K}}$ and $\boldsymbol{\beta}^{(i)} = \col(\beta^{(i)}_k)_{k \in \mathcal{K}}$, with $i \in \mathcal{N}$, the solution is unique, e.g., singled out using a convex tie-break rule \cite{posteriori-support}. 
\end{assumption}
We denote the unique optimal solution of \eqref{eq:PLM-N-prime} as $(\mathbf{x}^*_{\mathcal{N}}, \mathbf{u}^*_{\mathcal{N}})$. When computed, the PLM is interested in knowing what safety guarantees it admits against unseen uncertainties. Such certificates are fundamentally connected with the concept of support constraints/support samples. Considering the reformulation \eqref{eq:PLM-N-prime}, each sample $i \in \mathcal{N}$ gives rise to the randomized constraint:
\begin{align}
    \mathcal{C}_i = \Biggl\{ \begin{bmatrix} \mathbf{x} \\ \mathbf{u} \end{bmatrix} \in \mathbb{R}^{3K} : & \ \mathbf{x}_k \in \tilde{\mathcal{X}}_k(b_0, \ell^{(i)}_k, \beta^{(i)}_k), \nonumber \\
   & u_k \geq \ell^{(i)}_k, b_k \leq \beta_k^{(i)}, \forall k \in \mathcal{K} \Biggr\}.
\end{align}
\begin{definition}[Support constraints/ samples] 
\label{def:support_constraints}
\vspace{0.2cm}
A constraint $\mathcal{C}_i$, with $i \in \mathcal{N}$, is a \textit{support constraint} if its removal changes the optimal solution, i.e., $(\mathbf{x}^*_{\mathcal{N}}, \mathbf{u}_{\mathcal{N}}^*) \neq (\mathbf{x}_{\mathcal{N} \setminus \{i\}}, \mathbf{u}_{\mathcal{N} \setminus \{i\}})$, where the right-hand side denotes the optimal solution obtained after removing the constraint $\mathcal{C}_i$ from program \eqref{eq:PLM-N-prime}. The sample that corresponds to a support constraint $\mathcal{C}_i$ is called a support sample.
\end{definition}

Support constraints/samples encode which samples of the data are required to reconstruct the optimal solution. However, there can be ill-defined cases where multiple copies of the same constraint accumulate on top of each other. To exclude such degenerate cases, we impose the following.
\begin{assumption} \label{assum:nondegeneracy} (Non-degeneracy \cite{CampiGaratti2008exact})
For every sample $(\tilde{\boldsymbol{\ell}}, \tilde{\boldsymbol{\beta}})$ of size $N$, the solution $(\mathbf{x}^*_\mathcal{N}, \mathbf{u}^*_\mathcal{N})$ coincides with probability $1$ with the solution that is obtained after eliminating all the constraints that are not of support.
\end{assumption}

Assumption \ref{assum:nondegeneracy} ensures that there are no multiple copies of the same constraint for different samples, since if that were the case, the solutions calculated only using the support constraints would differ from the original randomized solution. 
Consider now the probability of violation $\mathbb{V}(\mathbf{r}^*_\mathcal{N}, \mathbf{b}^*_\mathcal{N})$ and the cardinality of support constraints or support samples defined as: 
\begin{equation}
    \label{eq:support-constraint-cardinality}
    s_N=|\{i \in \mathcal{N}: \mathcal{C}_i \text{ is a support constraint}\}|.
\end{equation} 
Then, the following result  provides \emph{a priori} robustness certificates for constraint satisfaction:
\begin{lemma} 
    \label{lemma:apriori}
    Consider Assumptions 1, 2, and 3, the data-driven program \eqref{eq:PLM-N-prime} with $N > 2K$. Then, the following holds: 
    \begin{equation}
	\mathbb{P}^N\!\left\{ \mathbb{V}(\mathbf{r}_N^*, \mathbf{b}_N^*) \leq \varepsilon \right\} \geq 1 - \underbrace{\sum_{i=0}^{2K-1} \binom{N}{i} \varepsilon^i (1 - \varepsilon)^{N-i}}_{=:\delta} \label{eq:beta_formula} 
    \end{equation}
    with $\varepsilon \in [0,1]$ being a violation level upper bound set by the PLM.
\end{lemma} 
\begin{proof}
    The reformulation \eqref{eq:PLM-N-prime} has two uncertain constraints, i.e., $b_k \leq \beta^{(i)}_k$ and $u_k \geq \max_{i \in \mathcal{N}} \ell^{(i)}_k$. Note that in \eqref{eq:PLM-N-prime}, there are $K$ unconstrained directions in the feasible space constructed by such constraints. As such, based on \cite[Def. 3.3]{SchildbachFagianoMorari2013}, we can replace the original bound on the total dimension with the support rank $d_{sr} = 2K$ and apply \cite[Lemma 3.8]{SchildbachFagianoMorari2013} and \cite[Thm. 4.1]{SchildbachFagianoMorari2013} on problem \eqref{eq:PLM-N-prime} to obtain:
    \begin{align*}
        \mathbb{P}^N \!\! \left\{\mathbb{V}(\mathbf{r}^*_\mathcal{N}, \mathbf{b}^*_\mathcal{N}) \leq \epsilon \right\} & = \mathbb{P}^N \left\{\mathbb{P}\left\{(\tilde{\boldsymbol{\ell}}, \tilde{\boldsymbol{\beta}}) : \begin{matrix} \exists k \in \mathcal{K}: \\ u^*_k < \tilde{\ell}_k \\ b^*_k > \tilde{\beta}_k \end{matrix} \right\} \leq \epsilon \right\} \\ &\geq 1 - \delta 
    \end{align*}
    where $\delta$ satisfies (\ref{eq:beta_formula}).
\end{proof}

As such, the PLM ensures that, with high confidence at least $1-\delta$, the probability that the state-of-charge of the virtual battery will be enough for future EV departures and state of charge caps is bounded by a violation level $\epsilon$ as stated in \eqref{eq:beta_formula}. Note that for a longer horizon $K$, the bound on $\epsilon$ becomes looser. This implies that, for planning further into the future,  the PLM requires a larger data size $N$ to ensure the safety of the constraint up to this violation level. 
The result of Lemma \ref{lemma:apriori} is \textit{a priori} in the sense that to guarantee these safety margins, the PLM is not required to know the samples beforehand. However, \emph{a priori} results often correspond to a worst-case bound and are thus conservative.

\subsection{Profit vs Risk in EV Parking Lot Management}
The results established in the previous subsection do not take into account the more general setting where the PLM decides for itself on the trade-off between constraints' satisfaction and profit. We grant this flexibility for the PLM by extending the problem in \eqref{eq:PLM-N-prime} as follows:
\begin{equation}
    \label{eq:PLM-N-rho-prime}
    \begin{aligned}
	\minimize_{\mathbf{u}, \boldsymbol{\xi}, \mathbf{x}} & \ J(\mathbf{r}) + \rho \sum_{i \in \mathcal{N}} \xi_i \\ 
	\st & \ \mathbf{x}_k \in \tilde{\mathcal{X}}_k\left(b_0, u_k - \xi_i, \xi_i + \beta^{(i)}_k\right), \ u_k \geq \ell^{(i)}_k,  \\ 
    & \ \forall k \in \mathcal{K}, \forall i \in \mathcal{N}
    \end{aligned}
\end{equation}
with $\boldsymbol{\xi} = \col(\xi_i)_{i \in \mathcal{N}} \in \mathbb{R}^N_{\geq 0}$.
From the PLM's perspective, (\ref{eq:PLM-N-rho-prime}) allows small, penalized violations of the training scenarios so that it can purposefully trade a degree of operational safety for lower cost or higher profit.
In this setting, the PLM has an additional penalty term in the cost function that is responsible for tightening or relaxing the safety constraints depending on the value of the weight $\rho \in \mathbb{R}_{\geq 0}$. Note that \eqref{eq:PLM-N-rho-prime} $\rightarrow$ \eqref{eq:PLM-N-prime} as $\rho \rightarrow \infty$. A positive $\xi_i$ generates the \textit{regret} associated with non-exact satisfaction of the associated $i$-th constraint of program \eqref{eq:PLM-N-rho-prime}.

Given that the data on vehicle arrivals and departures and the virtual state of charge bounds are available to the PLM, one can indeed obtain (possibly) tighter guarantees by leveraging the so-called \emph{a posteriori} bounds \cite{posteriori-support}, \cite{Campi_2020_scenario_relaxation}, \cite{JMLR_Campi}. To this end, we impose the following assumption:  
\begin{assumption}[Non-accumulation] 
    \label{assum:nonaccum}
    For every decision $(\mathbf{r}, \mathbf{b}, \mathbf{u})$ we have that: 
    \begin{equation}
        \mathbb{P}[ \exists k \in \mathcal{K} \mid b_{k}=b_{k-1}+q_k+ r_k-\ell_k \text{ or } b_k=\beta_k]=0
    \end{equation}
\end{assumption}
Assumption \ref{assum:nonaccum} avoids degenerate cases, where different samples lead to constraints that overlap at the solution. 
Choosing different weight parameters $\rho$, the PLM can assess the trade-off between performance and the state-of-charge constraints violation based on the following result, based on the scenario optimization with relaxation \cite{Campi_2020_scenario_relaxation}, \cite{JMLR_Campi}: 
\begin{proposition} 
    \label{prop:relaxation}
    Consider Assumptions \ref{assum:energy-prices}, \ref{assum:uniqueness} and \ref{assum:nonaccum}. Let $(\mathbf{u}_{ \rho}^*, \boldsymbol{\xi}^*_{\rho}, \mathbf{x}^*_{\rho})$ be the solution of \eqref{eq:PLM-N-rho-prime}.
    Given a confidence parameter $\delta \in (0,1)$, for any $m = 0, \dots, N-1$, $m < N$, consider the polynomial equation 
    \begin{multline} 
        \label{cor:apost_equation}
        \binom{N}{m} t^{\,N-m}
        - \frac{\delta}{2N} \sum_{i=m}^{N-1} \binom{i}{m} t^{\,i-m} \\
        - \frac{\delta}{6N} \sum_{i=N+1}^{4N} \binom{i}{m} t^{\,i-m}
        = 0 
    \end{multline}
    with respect to $t$.
    For $m = N$, consider instead the polynomial equation
    \begin{equation}
    1 - \frac{\delta}{6N} \sum_{i=N+1}^{4N} \binom{i}{N} t^{\,i-N} = 0. \label{cor:apost_equation2} 
    \end{equation}
    For any $m = 0, \dots, N-1$, equation (\ref{cor:apost_equation}) has exactly two solutions in 
    $[0, +\infty)$, which we denote by $t^{(m)}$ and $\bar{t}^{(m)}$, with 
    $t^{(m)} \leq \bar{t}^{(m)}$.  
    Instead, equation (\ref{cor:apost_equation2}) has only one solution in $[0, +\infty)$, which we denote 
    by $\bar{t}^{(N)}$, while we define $t^{(N)} := 0$. Let $\underline{\epsilon}(m) := \max\{0, 1 - t^{(m)}\},  
    \overline{\epsilon}(m) := 1 - \bar{t}^{(m)}$, for $m = 0,\dots,N$. Then, with confidence at least $1-\delta$, it holds that
    \begin{equation}
    \underline{\epsilon}(\tilde{s}_N^*) 
    \;\leq\; \mathbb{V}(\mathbf{r}_{\rho}^*, \mathbf{b}^*_{ \rho}) 
    \;\leq\; \overline{\epsilon}(\tilde{s}_N^*) \nonumber
    \end{equation}
    where 
    $\tilde{s}_N^*$ is the cardinality of samples $(\boldsymbol{\ell}^{(i)}, \boldsymbol{\beta}^{(i)})$ for which there exists $k \in \mathcal{K}$  such that $b_k^* \leq  b_{k-1}^* + q_k + r_k^* - \ell_k^{(i)}$
     or $b_k^{*} \geq \beta_{k}^{(i)}$.
\end{proposition}
\begin{proof} 
Consider the unique optimizer of \eqref{eq:PLM-N-rho-prime}. Then, the following equality holds:
   \begin{align*}
        &\mathbb{P}^N \!\! \left\{\mathbb{V}(\mathbf{r}^*_\rho, \mathbf{b}^*_\rho) \in  [\underline{\epsilon} (\tilde{s}^\ast_N), \overline{\epsilon} (\tilde{s}^\ast_N)]\right\} \\
        &= \mathbb{P}^N \left\{\mathbb{P}\left\{(\tilde{\boldsymbol{\ell}}, \tilde{\boldsymbol{\beta}}) : \begin{matrix} \exists k \in \mathcal{K}: \\ u^*_{k, \rho} < \tilde{\ell}_k \text{ or} \\ b^*_{k, \rho} > \tilde{\beta}_k  \end{matrix} \right\} \in [\underline{\epsilon} (\tilde{s}^\ast_N), \overline{\epsilon} (\tilde{s}^\ast_N)] \right\}. 
    \end{align*}
    Calculating $\tilde{s}^\ast_N$ 
    for \eqref{eq:PLM-N-rho-prime} and applying Theorem 2 in \cite{posteriori-support} concludes the proof.
\end{proof}
For the PLM, this means that the observed training sample can be used to derive a tighter, \emph{a posteriori} estimate of the true violation risk, thereby quantifying the profit/risk trade-off induced by the relaxation parameter $\rho$. Note that in Proposition \ref{prop:relaxation}, the constraints that are important and affect the quality of the generalization guarantees are the active or violating constraints depicted by the cardinality $\tilde{s}_N^*$. The smaller the cardinality  $\tilde{s}_N^*$, the tighter the lower and upper violation levels on the true probability risk. 

In practice, the data used by the PLM can be subject to manipulations either by noise or by adversarial entities. With the advent of cyber-physical systems, cyber-attacks on subsystems of the electricity grid have become increasingly more common. Furthermore, even after appropriate pre-processing, the PLM might not be sure how much trust to put into their data. In the following Section, we aim to address this issue by proposing a distributionally and adversarially robust virtual energy storage methodology for the PLM with tunable guarantees.

\section{Distributionally Robust Virtual Storage Services}
In the first part of this Section, we design the PLM such that it is robust against adversarial or noisy changes in the data. Specifically, following the approach in \cite{Campi2025_DRO}, we consider the case where the PLM  may trust the data up to a certain threshold that defines a \textit{trust-region}. Specifically, we consider that each sample $(\boldsymbol{\ell}, \boldsymbol{\beta})$ lies within an adversarial region $\mathcal{A}_{\boldsymbol{\ell}, \boldsymbol{\beta}} \subseteq \mathbb{R}^{2K}$ defined as 
\begin{align} 
    \mathcal{A}_{\boldsymbol{\ell}, \boldsymbol{
    \beta}}=\{(\boldsymbol{\ell}+\Delta \boldsymbol{\ell}, \boldsymbol{\beta} + \Delta\boldsymbol{\beta}): (\Delta \boldsymbol{\ell}, \Delta\boldsymbol{\beta})  \in \mathcal{A} \}, 
\end{align}
where $\mathcal{A} \subset \mathbb{R}^{2K}$ is considered to be a set of possible data deviations considered by the PLM as design choices. As data perturbations can alter the resulting decision and thus its associated robustness certificates, the notion of the probability of violation of the virtual SoC constraints has to be redefined to account for a risk function robust against not only the drawn sample but also a region around it. If mistrust is high, the region is larger, while if the PLM trusts the measurements, the region is smaller. 
Consider the adversarial region $\mathcal{A}_{\boldsymbol{\ell}, \boldsymbol{\beta}}$ around sample $(\boldsymbol{\ell}, \boldsymbol{\beta})$ and that $(\mathbf{r}, \mathbf{b})$ is the PLM's decision. Then, the risk measure is defined as 
\begin{multline}
\mathbb{V}_A(\mathbf{r},\mathbf{b}) 
= \mathbb{P}\Big\{ (\boldsymbol{\ell}, \boldsymbol{\beta}) : 
\exists (\tilde{\boldsymbol{\ell}}, \tilde{\boldsymbol{\beta}}) \in \mathcal{A}_{\boldsymbol{\ell}, \boldsymbol{\beta}}, k \in \mathcal{K} : \\
 b_k < b_{k-1} + q_k + r_k - \tilde{\ell}_k \text{ or } b_k > \tilde{\beta}_k  
\Big\} \nonumber
\end{multline}
is called the \textit{adversarial probability of violation}.
Such a risk measure does not consider only the particular points to account for violations, but an entire set of perturbations around the data belonging to $\mathcal{A}_{\boldsymbol{\ell}, \boldsymbol{\beta}}$. 
In this paper, we consider a finite approximation of $\mathcal{A}_{\boldsymbol{\ell}, \boldsymbol{\beta}}$ obtained as the convex hull of $M \in \mathbb{N}$ points in $\mathcal{A}$ and denoted by $\hat{\mathcal{A}} \subseteq \mathcal{A}$. For ease of notation, we denote $\mathcal{M}=\{1, \dots, M\}$. Then, the adversarially robust optimization program that the PLM aims to solve takes the form: 
\begin{equation}
    \label{eq:PLM-N-rho-A-prime}
    \begin{aligned}
	\minimize_{\mathbf{u}, \boldsymbol{\xi}, \mathbf{x}} & \ J(\mathbf{r}) + \rho \sum_{i \in \mathcal{N}} \xi_i \\ 
	\st & \ \mathbf{x}_k \in \tilde{\mathcal{X}}_k\left(b_0, u_k \! - \! \xi_i, \xi_i \! + \! \beta^{(ij)}_k\right), \ u_k \geq \ell^{(ij)}_k, \\ & \ \forall k \in \mathcal{K}, \forall i \in \mathcal{N}, \forall j \in \mathcal{M},
    \end{aligned}
\end{equation}
where $\mathcal{M}$ is the set of points of $\mathcal{A}$ used to approximate $\mathcal{A}_{\boldsymbol{\ell}, \boldsymbol{\beta}}$.  
Using the formulation in (\ref{eq:PLM-N-rho-A-prime}) the PLM can choose a schedule that stays robust not only for the recorded data, but also for nearby corrupted, noisy, or adversarial versions of that data.

Note that the data points of the adversarial set that violate the constraints would correspond to the empirical adversarial risk of the PLM. However, considering only those samples as support samples would not be enough to assess the data-driven decision's out-of-sample performance. To account for potential overfitting issues, we also need to consider the points that lead to active constraints on the boundary of the PLM's feasible set formed by sampled constraints. In accordance with \cite[Def. 5]{Campi2025_DRO}, a sample $(\boldsymbol{\ell}^{(i)}, \boldsymbol{\beta}^{(i)})$ is an \textit{adversarial support sample} or contributes to the \emph{adversarial complexity} for problem \eqref{eq:PLM-N-rho-A-prime} for some time step $k \in \mathcal{K}$ if one of the following conditions holds:
\begin{align*}
    \exists (\tilde{\boldsymbol{\ell}}^{(i)}, \tilde{\boldsymbol{\beta}}^{(i)}) &\in \mathcal{A}_{\boldsymbol{\ell}^{(i)}, \boldsymbol{\beta}^{(i)}}: 
        u^*_{k, \hat{\mathcal{A}}} < \tilde{\ell}^{(i)}_k \text{ or } b^*_{k, \hat{\mathcal{A}}} > \tilde{\beta}^{(i)}_k  \\
    \exists (\tilde{\boldsymbol{\ell}}^{(i,j)}, \tilde{\boldsymbol{\beta}}^{(i,j)}) &\in \mathcal{A}_{\boldsymbol{\ell}^{(i,j)}, \boldsymbol{\beta}^{(i,j)}}: 
        u^*_{k, \hat{\mathcal{A}}} = \tilde{\ell}^{(i,j)}_k \!\! \text{ or } 
        b^*_{k, \hat{\mathcal{A}}} = \tilde{\beta}^{(i,j)}_k \\
    \exists (\tilde{\boldsymbol{\ell}}^{(i,j)}, \tilde{\boldsymbol{\beta}}^{(i,j)}) &\in \mathcal{A}_{\boldsymbol{\ell}^{(i,j)}, \boldsymbol{\beta}^{(i,j)}}: u^*_{k,\hat{\mathcal{A}}} < \tilde{\ell}^{(i,j)}_k \!\! \text{ or } b^*_{k, \hat{\mathcal{A}}} > \tilde{\beta}^{(i,j)}_k
\end{align*}
where $u_{k,\hat{\mathcal{A}}}^*$ denotes the optimal value of the auxiliary variable of \eqref{eq:PLM-N-rho-A-prime}.

\begin{proposition} \label{thrm:adv_guarantees}
    Under Assumptions \ref{assum:energy-prices}, 2 and  \ref{assum:nonaccum} and the condition $\hat{\mathcal{A}}_{\boldsymbol{\ell}, \boldsymbol{\beta}} \subseteq \mathcal{A}_{\boldsymbol{\ell}, \boldsymbol{\beta}}$ for all $(\boldsymbol{\ell}, \boldsymbol{\beta}) \in (\mathcal{L} \times \mathcal{B})^K$, it holds that with high confidence at least $1-\delta$:
    \begin{equation}
        \mathbb{V}_\mathcal{A}(\mathbf{r}^*_{\hat{\mathcal{A}}}, \mathbf{b}^*_{\hat{\mathcal{A}}}) \in \left[\underline{\epsilon}(s^*_{\mathcal{A}, \hat{\mathcal{A}}}), \overline{\epsilon}(s^*_{\mathcal{A}, \hat{\mathcal{A}}}) \right], 
    \end{equation} 
    where $s^*_{\mathcal{A}, \hat{\mathcal{A}}}$ denotes the number of adversarial support samples of \eqref{eq:PLM-N-rho-A-prime} and $\underline{\epsilon}(m)$ and $\overline{\epsilon}(m)$ are obtained by solving the polynomial equations of Proposition \ref{prop:relaxation} and then setting $m=s^*_{\mathcal{A}, \hat{\mathcal{A}}}$.
\end{proposition}
\begin{proof} 
    The following equalities hold:
    \begin{align*}
        \mathbb{V}_A(\mathbf{r},\mathbf{b}) &= \mathbb{P}\big\{(\boldsymbol{\ell}, \boldsymbol{\beta}): \exists \tilde{\boldsymbol{\ell}}, \tilde{\boldsymbol{\beta}} \in \mathcal{A}_{\boldsymbol{\ell}, \boldsymbol{\beta}}, k \in \mathcal{K}: \\ 
        & \qquad b_k > \tilde{\beta}_k \text{ or }b^*_k < b^*_{k-1}+q_k+r_k-\tilde{\ell}_k  \big\} \\ 
        & = \mathbb{P}\{(\boldsymbol{\ell}, \boldsymbol{\beta}): \exists\tilde{\boldsymbol{\ell}}, \tilde{\boldsymbol{\beta}} \in \mathcal{A}_{\boldsymbol{\ell}, \boldsymbol{\beta}}: \\ 
        & \qquad \max_{k \in \mathcal{K}}\max\{u^*_{k}-\tilde{\ell}_k,  b^*_k - \tilde{\beta}_k\}  \} >0 \} \\
        &=\mathbb{P}\{(\boldsymbol{\ell}, \boldsymbol{\beta}): \exists \tilde{\delta} \in \mathcal{A}_{\boldsymbol{\ell}, \boldsymbol{\beta}}: f(\mathbf{r},\mathbf{b},\mathbf{u}, \tilde{\boldsymbol{\ell}}, \tilde{\boldsymbol{\beta}})>0 \} 
    \end{align*}
    where function $f$ in the last equality is defined as
    \begin{equation}
        f(\mathbf{r},\mathbf{b},\mathbf{u}, \tilde{\boldsymbol{\ell}}, \tilde{\boldsymbol{\beta}})=\max\limits_{k \in \mathcal{K}}\max\{u_k-\tilde{\ell}_k,  b_k - \tilde{\beta}_k\} 
    \end{equation}
    Then, the result follows from \cite[Thm. 3]{Campi2025_DRO}.
\end{proof}
As such, the PLM can still rely on robustness certificates even when each training sample is mistrusted within a prescribed perturbation set, thus protecting the policy against adversarial or noisy data.

\begin{algorithm}[t]
\caption{Tunable Adversarially Robust VESS}
\begin{algorithmic}[1]
\Require $r_{\max}$; $\boldsymbol{\ell}^{(i)},\boldsymbol{\beta}^{(i)}$ for $i=1,\dots,N$, $j=1,\dots,M$; $\mu$, data perturbation set $\mathcal R$, $M$, $\delta$, update rule $\mathcal U$, bounds $N_{\max},\rho_{\max}$, tolerance $\tau$, $S_{\max}$
\State \textit{Retailer--PLM}: Contract agreement on selling and buying price $(\pi_k^{+},\pi_k^{-})$ per time step $k\in\mathcal K$.
\State \textit{PLM--Prosumers}: Contract agreement on energy request $q_k$ per time step $k\in\mathcal K$.
\State \textit{PLM}: Fix $\epsilon_{\text{goal}}$, and set $\epsilon \gets \infty$, $\epsilon_{\mathrm{prev}} \gets \infty$, $s \gets 0$
\While{$\epsilon > \epsilon_{\text{goal}}$}
    \State $\epsilon \gets \infty$
    \ForAll{$R \in \mathcal R$}
        \State Solve \eqref{eq:PLM-N-rho-A-prime} and obtain $\mathbf{x}^*_{\hat{\mathcal{A}}}(R)$
        \State Compute the adversarial complexity $s^*_{\mathcal{A},\hat{\mathcal{A}}}(R)$
        \State Compute $\epsilon(R) \gets \overline{\epsilon}\big(s^*_{\mathcal{A},\hat{\mathcal{A}}}(R)\big)+\mu/R$
        \If{$\epsilon(R) < \epsilon$}
            \State $\epsilon \gets \epsilon(R)$, $R^* \gets R$
            \State $(\mathbf{r}^*_{\text{safe}}, \boldsymbol{b}^*_{\text{safe}}) \gets (\mathbf{r}^*_{\hat{\mathcal{A}}}(R), \boldsymbol{b}^*_{\hat{\mathcal{A}}}(R))$
        \EndIf
    \EndFor
    \State $s \gets \begin{cases}
    s+1, & \text{if } |\epsilon_{\mathrm{prev}}-\epsilon|\le \tau\\
    0, & \text{otherwise}
    \end{cases}$
    \State $\epsilon_{\mathrm{prev}} \gets \epsilon$
    \If{$s \ge S_{\max}$ \textbf{or} $(N,\rho)=(N_{\max},\rho_{\max})$}
        \State \textbf{break}
    \EndIf
    \State $(N,\rho) \gets \mathcal U(N,\rho)$
\EndWhile
\State \Return $(\mathbf{r}^*_{\text{safe}},\, \boldsymbol{b}^*_{\text{safe}},\, \epsilon,\, R^*)$
\end{algorithmic}
\end{algorithm}

\subsection{EV Departures and Capacity Distributional Shifts}
The training set $\{\boldsymbol{\ell}^{(i)}, \boldsymbol{\beta}^{(i)}\}_{i \in \mathcal{N}}$ can originate from, e.g., synthetic models or real-world measurements/ historical data. However, due to distributional shifts, the data used for training might not follow the same distribution as the data collected after deployment. In our setting, the PLM has collected data on the losses incurred from vehicle departures and data of imposed upper bounds in the state of charge of the virtual storage. However, the PLM wishes to have safety certificates against yet unseen realizations of these quantities that might follow a different distribution $\mathbb{P}'$. Unfortunately, without some connection between the probability distributions $\mathbb{P}$ and $\mathbb{P}'$, it is extremely challenging to provide any provable guarantees. To establish such results, a metric of similarity among distributions is often considered. An often used measure, due to its intuitive interpretation based on optimal transport, is the so-called Wasserstein distance defined as follows: 

\begin{definition}
Consider the uncertain parameters $(\ell, \beta)$ and $(\tilde{\ell}, \tilde{\beta})$ following the probability distributions $\mathbb{P}$ and $\mathbb{P}'$. 
Then, the Wasserstein metric is given by: 
\begin{equation}
d_W(\mathbb{P}, \mathbb{P}')=\inf_{\mathbb{Q}} \mathbb{E}_\mathbb{Q}[\|\ell-\tilde{\ell}\|_2+|\beta-\tilde{\beta}\|_2],
\end{equation}
where $\mathbb{Q}$ denotes a joint probability distribution of random variables with marginals $\mathbb{P}$ and $\mathbb{P}'$. \hfill $\square$
\end{definition}
Based on the Wasserstein distance, we can then define the ambiguity set that the PLM selects to account for risk aversion against probabilistic shifts with respect to this metric. To achieve this, the PLM needs to decide on a radius, which determines the size of the ambiguity set. As such, we assume that the distance between the training data distribution $\mathbb{P}$ and the test data distribution $\mathbb{P}'$ is coupled via the inequality $d_W(\mathbb{P}, \mathbb{P}') \leq \mu$ for some PLM defined Wasserstein radius $\mu>0$. The ambiguity set defined based on $d_W$, is defined as $\mathbb{B}_{\mu}(\mathbb{P})=\{\mathbb{P}': d_W(\mathbb{P}, \mathbb{P}') \leq \mu \}$.
We define the probability of violation for a distribution $\mathbb{P}'  \in \mathbb{B}_{\mu}(\mathbb{P})$ as:
\begin{equation}
\mathbb{V}'_\mathcal{A}(\mathbf{r}, \mathbf{b}):= \mathbb{P}'\{(\boldsymbol{\ell}', \boldsymbol{\beta}'): f(\mathbf{r},\mathbf{b},\boldsymbol{u}, \boldsymbol{\ell}', \boldsymbol{\beta}') > 0\}
\end{equation}
Furthermore, we assume that $\|\boldsymbol{\ell}_K-\tilde{\boldsymbol{\ell}}_K\|_2 \leq R_\ell$ and $\|\boldsymbol{\beta}_K-\tilde{\boldsymbol{\beta}}_K\|_2 \leq R_\beta$, where $R_\ell, R_\beta \in \mathbb{R}_{\geq 0}$ are choices of the designer that determine how much they trust the possible realizations of the uncertainty obtained from distributions within the ambiguity set. In this setting, the following holds:
\begin{theorem} 
    \label{thrm:dist_adversarial_bound} 
    Consider Assumptions \ref{assum:energy-prices}, 2 and \ref{assum:nonaccum} and $\mathbb{P}' \in \mathbb{B}_{\mu}(\hat{\mathbb{P}})$. Then, with confidence at least $1-\delta$:
    \begin{equation}
        \mathbb{V}'_\mathcal{A}(\mathbf{r}^*_{\hat{A}}, \mathbf{b}^*_{\hat{A}}) 
           \leq \overline{\epsilon}(s^*_{A, \hat{A}})
           + \frac{\mu}{R}
    \end{equation}
    where $R=R_\beta+R_\ell$ and $s^*_{\mathcal{A}, \hat{\mathcal{A}}}$ is the adversarial complexity. 
\end{theorem}
\emph{Proof:} The proof follows by the equivalence of the out-of-distribution risks of \eqref{eq:PLM-N-rho-A-prime} and the application of Theorem 5 in \cite{Campi2025_DRO}. \hfill $\blacksquare$

From a practical standpoint, Theorem \ref{thrm:dist_adversarial_bound} tells the PLM how much out-of-distribution risk it can certify when future EV behaviour differs from the training data, while accounting for both sample mistrust and distributional shift. Note that being more risk-averse towards probabilistic shifts by increasing $\mu$ results in a looser bound. While a larger $R$ would seemingly improve the bound, this is not necessarily the case, as a larger $R$ can lead to a larger number of adversarial support samples, which can then worsen the guarantees. As such, the risk-aversion of the PLM will have a direct effect on the quality of the theoretical guarantees they can provide.
\begin{remark}
In summary, the parameters $\rho$, $R$, and $\mu$ admit a direct operational interpretation for the PLM. Specifically, $\rho$ reflects the desired profit-vs-safety preference, $R$ reflects the amount of mistrust in the training samples due to noise, forecasting error, or possible corruption, and $\mu$ quantifies the level of protection against distributional shift around the empirical distribution. In practice, these quantities can be calibrated from historical variability and validation data, and then selected over a finite grid of options. 
\end{remark}

Algorithm~1 describes a general methodology for computing a certified virtual energy storage policy for the PLM. For a fixed target violation level $\epsilon_{\mathrm{goal}}$, the PLM solves~\eqref{eq:PLM-N-rho-A-prime} over a finite set of candidate trust radii $R\in\mathcal R$. For each $R$, it computes the corresponding adversarial complexity $s_{\mathcal A,\hat{\mathcal A}}^\star(R)$ and evaluates the certificate $\overline{\epsilon}(s_{\mathcal A,\hat{\mathcal A}}^\star(R))+\mu/R$, as implied by Theorem~1. The radius yielding the smallest certificate is selected, and the associated solution is stored as the current best certified policy. If the target level is not met, the PLM updates the tuning parameters through a designer-defined rule $\mathcal U$, which may increase $N$, $\rho$, or both, i.e.,  
$\mathcal U(N,\rho)=(\min\{N+N^+,N_{\max}\},\rho)$,
$\mathcal U(N,\rho)=(N,\min\{\rho+\rho^+,\rho_{\max}\})$,
or
$\mathcal U(N,\rho)=(\min\{N+N^+,N_{\max}\},\min\{\rho+\rho^+,\rho_{\max}\})$, respectively. The procedure stops when the desired certificate is reached, when the improvement becomes negligible for several consecutive iterations, or when the admissible budgets on $N$ and $\rho$ are exhausted. In all cases, Algorithm 1 returns the best certified policy found so far, together with its certificate and selected radius $R^\ast$.

\section{Numerical Study}
\begin{figure}[t] 
  \centering
  \includegraphics[width=0.49\textwidth]{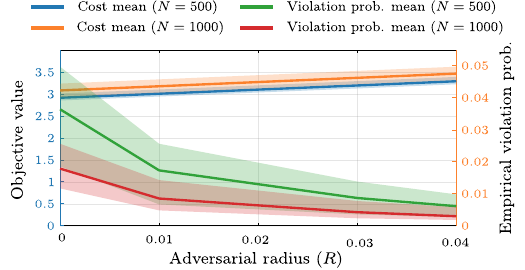}
  \caption{\small Trade-off study between adversarially robust empirical probability of violation vs the profit of the PLM for varying values of $R$.}
  \label{fig:trade-off}
\end{figure}

\begin{figure}[t] 
  \centering
  \includegraphics[width=0.49\textwidth]{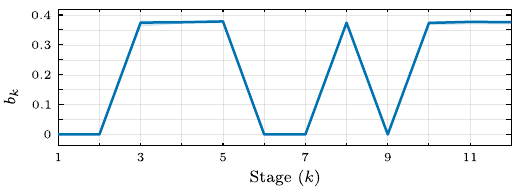}
  \caption{\small Optimal virtual state of charge $b_k$ of the PLM's energy buffer at each time step $k$ for  $R=0.01$ and $\rho=1$.}
  \label{fig:b}
\end{figure}

\begin{figure}[t] 
  \centering
  \includegraphics[width=0.49\textwidth]{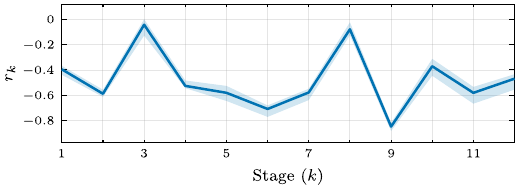}
  \caption{\small Optimal energy $r_k$ sold to the retailer at each time step $k$ for  $R=0.01$ and $\rho=1$.}
  \label{fig:r}
\end{figure}

\begin{figure}[t] 
  \centering
  \includegraphics[width=0.4\textwidth]{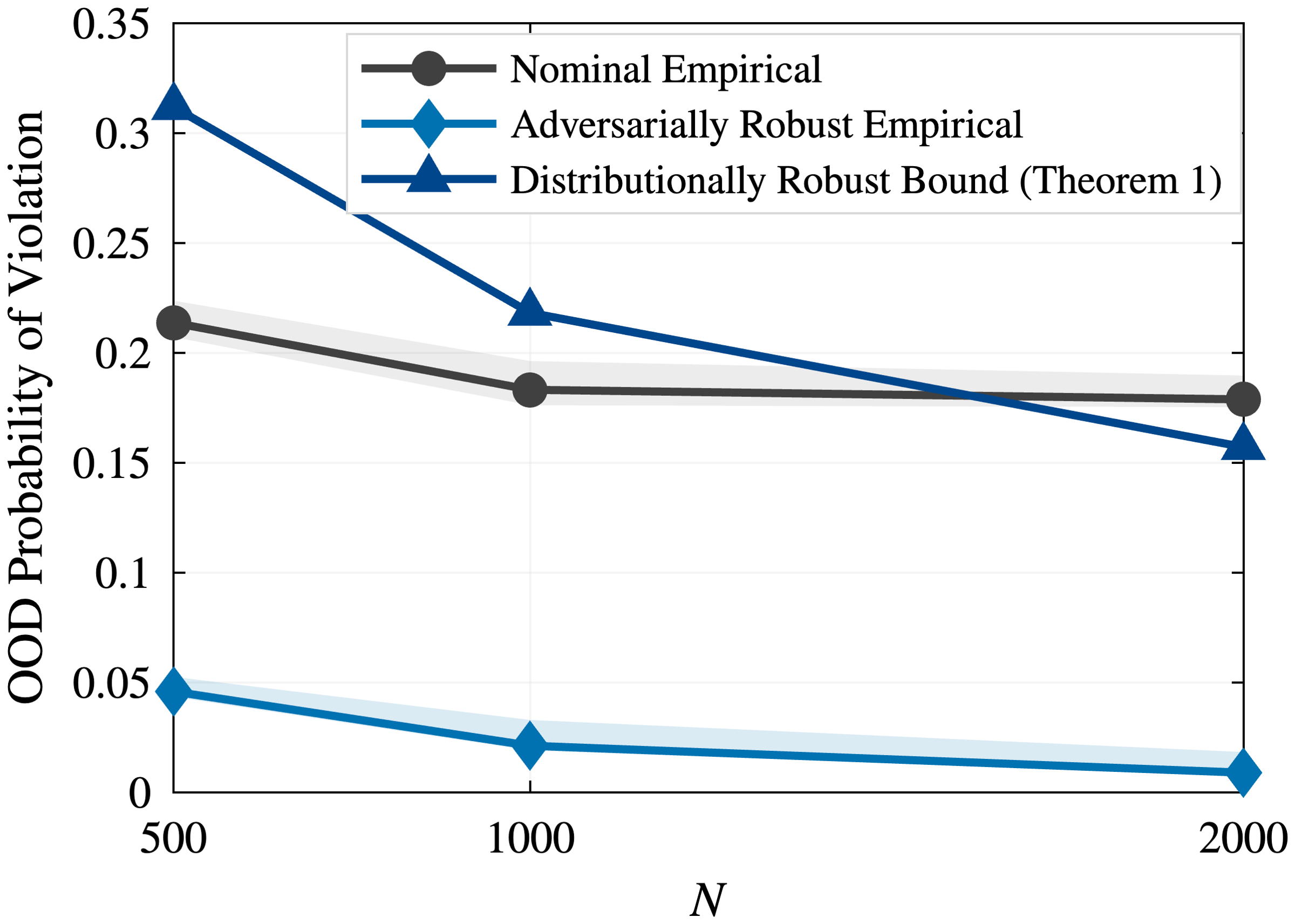}
  \caption{OOD violation probabilities show that the adversarially robust policy consistently remains below the certified bound of Thm. \ref{thrm:dist_adversarial_bound}, while the nominal policy exhibits significantly poorer safety and exceeds the certified risk level achieved by the robust policy at $N=2000$.}
  \label{fig:OOD}
\end{figure}

We consider training data generated entrywise as $\ell_{\mathrm{nom}} \sim 0.1\,\mathcal{N}(0,1)$ and $\beta_{\mathrm{nom}} \sim 0.4 + 0.5\,\mathcal{U}[0,1]$. We then generate $j=1,\dots,6$ perturbed scenarios as $\ell^{(j)}=\ell_{\mathrm{nom}}+s_\ell\,\mathcal{N}(0,1)$ and $\beta^{(j)}=\max\!\left\{0,\beta_{\mathrm{nom}}+s_\beta\,\mathcal{N}(0,1)\right\}$, where $s_\ell=s_\beta=0.01$. In the robust formulation, $R$ acts as an additional trust-radius parameter reflecting the PLM's level of mistrust in the training data.  In Figure \ref{fig:trade-off}, $R$ varies in $[0,0.04]$ for parametric analysis, while for Figures \ref{fig:b} and \ref{fig:r} it is fixed at $R=0.01$. For Figure \ref{fig:OOD}, $R$ is selected differently, i.e., by grid search over candidate trust radii as in Algorithm~1, while $\mu$ and $\rho$ should be interpreted as design parameters that, in practice, would be calibrated from historical inter-day variability and the operator's preferred profit-risk profile. We set these values to $\mu=10^{-3}$ and $\rho=1$, respectively. The exogenous prosumer request and retailer prices are random but fixed for each simulation, given by
$q_k = 0.2\sin\!\left(\frac{k}{4}\right) + 0.1\,w_k, w_k \sim \mathcal{N}(0,1)$, while $\pi_k^{+}$, $\pi_k^{-}$ are also fixed per time step according to
$\pi_k^{+} = 1 + u_k^{+}$, and 
$\pi_k^{-} = -1 + 0.5\,u_k^{-}$, where $u_k^{+}, u_k^{-} \sim \mathcal{U}[0,1]$.

The objective function of the PLM and the probability of violation for sample sizes $N \in \{500, 1000\}$ and varying values of the radius $R$ is shown in Figure \ref{fig:trade-off}. Note that for a higher number of samples, the probability of violation improves significantly at the expense of a slightly higher cost value. 
The horizon is fixed at $K=12$ time steps, and the energy $r_k$ bought/sold from/to the retailer at each time step $k$, is bounded by $r_{\max}=5$.

Figures \ref{fig:b} and \ref{fig:r} illustrate the virtual state of charge $b_k$ of the PLM's energy buffer and the energy $r_k$ sold to the retailer at each time step $k$ for different multi-samples. To evaluate the out-of-distribution (OOD) performance we consider a test data set $(\boldsymbol{\ell}^{(i)}, \boldsymbol{\beta}^{(i)})$, $i \in \{1, \dots, N_{\text{test}}\}$. The samples are obtained each time from $N'$ different probability distributions $\mathbb{P}_v, v \in \{1, \dots, N'\}$ obtained by perturbing the nominal probability distribution in different ways and then scaling them down such that they belong to the considered ambiguity set. We then wish to test the OOD violation level for each of those perturbed probability distributions. To do this, we calculate the corresponding empirical probability of violation defined as:
\begin{align}
\hat{\mathbb{V}}_{v}(\boldsymbol{\ell}, \boldsymbol{\beta})=\frac{1}{N_{\text{test}}} \sum_{i=1}^{N_{\text{test}}} \mathds{1}_{\{  \exists k \in \mathcal{K}: b^*_k <  b^*_{k-1} + q_k + r^\ast_k - \ell^{(i)}_k \text{ or } b^*_k > \beta^{(i)}_{k} \}}, \nonumber 
\end{align}
where $N_{\text{test}}=10^{4}$.
To see how well our model performs against probabilistic shifts, we use the empirical mean across the empirical probabilities of violation of $N'=40$ distributions within the considered ambiguity set. The results are summarized in Figure \ref{fig:OOD}, where $\delta=10^{-5}$ and a different number of samples $N \in \{500,1000, 2000\}$ is used. The theoretical OOD level is computed by testing $30$ positive values of $R$ and selecting the one minimizing $\overline{\varepsilon} \!\bigl(s^*(R)\bigr)+\mu/R$. Specifically, for the trust-radius selection, we test \(n_R=30\) logarithmically spaced candidate values in the interval \([3\mu,\,0.25]\), i.e., $\mathcal{R}=\{R_1,\dots,R_{30}\}\subseteq [3\mu,\,0.25]$.
The radius is then selected by grid search over \(\mathcal{R}\), and the confidence budget is split uniformly across the tested radii, namely \(\delta_R=\delta/30\). The OOD probability of violation of the nominal policy, i.e., the optimal PLM policy obtained without adversarial training, serves as an empirical benchmark, showing numerically that a solution trained without adversarial samples exhibits substantially worse performance under distribution shifts and exceeds the certified risk level achieved by the adversarially robust policy at $N=2000$. 

Finally, note that the theoretical OOD level is a high-confidence worst-case certificate over the entire ambiguity set, whereas the empirical OOD violation is computed only on
40 sampled perturbation models. Hence, the theoretical level is expected to be conservative and need not be numerically tight compared to the empirical OOD probability of violation.

\section{Conclusion}
This paper develops a distributionally robust framework based on scenario optimization that enables a parking-lot manager to operate aggregated EVs as a virtual energy storage system, providing profit/risk tuning flexibility and finite-sample guarantees under adversarial perturbations and Wasserstein distribution shifts. Numerical simulations of the proposed model show agreement between empirical violations and theoretical bounds. Future work will involve integrating user-centric EV battery health considerations into this scheme and modelling the EV users as active participants of the parking lot management system. Furthermore, this model can be extended to large-scale implementation by incorporating multiple interacting parking lots to model market participation and network constraints. Finally, we will focus on real-world deployment through data-driven estimation of EV departure distributions, realistic metering and forecasting error models, and online recalibration of the design parameters $\rho$, $R$, and $\mu$ to reflect evolving operating conditions.

\balance
\bibliographystyle{IEEEtran}
\bibliography{references} 

\end{document}